\documentclass[11pt,reqno]{amsart}
\usepackage[T1]{fontenc}
\usepackage[utf8]{inputenc}
\usepackage{lmodern}
\usepackage[a4paper,margin=1in]{geometry}
\usepackage{amsmath,amssymb,amsfonts,amsthm,mathtools}
\usepackage{enumitem,microtype,xcolor}
\usepackage[numbers,sort&compress]{natbib}
\usepackage{hyperref}
\allowdisplaybreaks[2]
\hypersetup{
 bookmarksdepth=2,
 colorlinks=true,
 linkcolor=blue!45!black,
 citecolor=blue!45!black,
 urlcolor=blue!45!black,
 pdftitle={Gaussian Conditional Spectra and Levy Area},
 pdfauthor={Guangqian Zhao},
 pdfsubject={Conditional Gaussian effective rank, exterior overlap, spectral dilation, and projective Brownian rough paths},
 pdfkeywords={Gaussian randomization, effective rank, exterior forms, overlap, concentration, Levy area, covariance entropy, projective rough paths}
}
\theoremstyle{plain}
\newtheorem{theorem}{Theorem}

\newtheorem{lemma}{Lemma}
\theoremstyle{remark}

\newcommand{\E}{\mathbb E}
\newcommand{\Pbb}{\mathbb P}
\newcommand{\Sp}{\mathcal S}
\newcommand{\tr}{\operatorname{tr}}
\newcommand{\Det}{\det{}_2}
\begin{document}
\title[Gaussian Conditional Spectra]{Gaussian Conditional Spectra and L\'evy Area}
\author{Guangqian Zhao}
\address{School of Mathematical Sciences, University of Science and Technology of China, Hefei, Anhui 230026, China}
\email{zhaoguangqian@mail.ustc.edu.cn}
\date{}

\begin{abstract}
For Gaussian images of Hilbert two-forms, a covariance formula gives
conditional spectral realizations, deletion interlacing, and positive row
responses. One random dilation compares all conditional spectra with the
mean model, with optimal universal weak-$L^1$ exponent. Row retention bounds all higher Gaussian chaoses; two low-order
input--output overlaps characterize determinant concentration by covariance
derivative matrices, equivalently by predictable responses. For fixed infinite-rank data the criterion is slow variation
of the covariance tail. For Wiener area we obtain exact spectral counting
asymptotics for every infinite-rank covariance, the Schatten classification,
and a necessary and sufficient covariance entropy condition for the
canonical projective Brownian rough path.
\end{abstract}
\maketitle

\section{The conditional spectral theorem}
\label{sec:exterior}

We study the spectrum of a Gaussian image of a Hilbert two-form through
its conditional determinant. The covariance observation \eqref{eq:generator}
has two readouts: rowwise convex order controls both spectral dilation
and higher Gaussian chaoses, while the first two slot derivatives detect
concentration. Theorem~\ref{thm:main} develops both; Theorem~\ref{thm:terminal-projective}
applies them to Wiener area and its projective lift.

The conditional realization uses real stability and mixed characteristic
polynomials \cite{Barvinok1999,BorceaBranden2009,MSS2015}.
For L\'evy area and rough paths see
\cite{Lyons1998,FrizHairer2020,LedouxLyonsQian2002,Dereich2010,GrongNilssenSchmeding2022};
for classical area formulas and finite-dimensional spectral results see
\cite{Levy1951,HelmesSchwane1983,Kree1986,FerreiroUtzet2011,DelaCruzOberhauser2026}.
Our covariance criteria complement weak-trace regularity
\cite{Zhao2026WeakTrace}. The projective-lift criterion addresses the
projective-area question in \cite[Chapter~4, abstract]{LyonsQian2002},
complementing \cite{Victoir2004}.

\paragraph{Notation and preliminaries.}
Hilbert spaces are real and separable. We identify $x\otimes y$ with
$h\mapsto x\langle y,h\rangle$; $\otimes$ also denotes the Hilbert
tensor product. For compact $T$, write $s_1(T)\ge s_2(T)\ge\cdots$
for its singular values, with multiplicities and zeros, and put
\[
 N_T(t)=\#\{j:s_j(T)>t\},\qquad
 \|T\|_p^p=\sum_js_j(T)^p\in[0,\infty],\qquad
 \Sp_p=\{T:\|T\|_p<\infty\}\quad(p>0).
\]
After complexification, $\det$ is the Fredholm determinant
and $\Det(I+T)=\det((I+T)e^{-T})$ for $T\in\Sp_2$ \cite{Simon2005}.

In $\Lambda^rE$ the increasing basis wedges are orthonormal, $\Lambda^0E=\mathbb R$,
$u^{\wedge0}=1$, and $\|u\wedge v\|\le\binom{r+s}{r}^{1/2}\|u\|\|v\|$
for degrees $r,s$. Exterior powers of bounded maps act on each slot. The map
$x\wedge y\mapsto x\otimes y-y\otimes x$ identifies $\Lambda^2E$
with skew-adjoint Hilbert--Schmidt operators. Denote the form of $T$
by $\omega_T$; thus $\|T\|_2^2=2\|\omega_T\|^2$.
If the nonzero eigenvalue pairs are $\pm i\sigma_j$, then
\begin{equation}\label{eq:det}
 Z_T(\lambda):=\Det(I+\lambda T)
 =\sum_{k\ge0}\left\|\frac{\omega_T^{\wedge k}}{k!}\right\|^2\lambda^{2k}
 =\prod_j(1+\lambda^2\sigma_j^2).
\end{equation}
Indeed, $\omega_T=\sum_j\sigma_j e_{2j-1}\wedge e_{2j}$ in suitable
orthogonal block coordinates. For $\lambda>0$, the resolvent trace is
\[
 \mathcal R_T(\lambda)
 =\tr[\lambda^2T^*T(I+\lambda^2T^*T)^{-1}]
 =\lambda\frac{d}{d\lambda}\log Z_T(\lambda).
\]

Let $(\Omega,\mathcal F,\Pbb)$ be a probability space. Let $H,K$ be Hilbert spaces, and let $Q\ge0$ be trace class on $H$. Let $G:K\to L^2(\Omega;H)$ be linear, with the scalar family
$\{\langle Gf,h\rangle:f\in K,\ h\in H\}$ centered and jointly Gaussian. Assume
\begin{equation}\label{eq:covariance}
 \E[Gf\otimes Gg]=\langle f,g\rangle Q,\qquad f,g\in K.
\end{equation}
Gaussian evaluation is the linear map
$f_1\wedge\cdots\wedge f_r\mapsto Gf_1\wedge\cdots\wedge Gf_r$
on algebraic wedges. No samplewise bounded operator $G:K\to H$ is assumed.
Write $(q_i)$ for the finite or countable sequence of positive eigenvalues of $Q$, counted with multiplicity. For integers $r\ge1$, set
\[
 S=\tr Q=\sum_iq_i,\qquad
 e_r(q)=\sum_{i_1<\cdots<i_r}q_{i_1}\cdots q_{i_r}
       =\tr(\Lambda^rQ),\qquad e_0(q)=1.
\]
Write $q_I=\prod_{i\in I}q_i$, $\tau_Q(t)=\sum_{q_i\le t}q_i$, and
$M_Q(t)=\sum_i\min(1,q_i/t)=N_Q(t)+\tau_Q(t)/t$.
An empty sum is zero and $r!e_r(q)\le S^r$.

Write $\xi_i=q_i^{-1/2}\langle G(\cdot),e_i\rangle$ in a covariance
basis, and $\mathcal F_J=\sigma(\xi_i:i\in J)$ for finite row sets $J$.
Fix a predictable exhaustive order: $i_j$ is chosen among unrevealed
indices from the observed history, and every positive-weight row is
eventually revealed almost surely. Put $J_0=\varnothing$,
$J_j=\{i_1,\ldots,i_j\}$, $\mathcal G_j=\sigma(i_a,\xi_{i_a}:a\le j)$,
$P_j=P_{J_j}$, $T_j=T_{J_j}$, with $P_*=Z_A$, $T_*=A$;
stage suprema include $*$. The reference order is $i_j=j$;
finite-rank rules stop after the last row.
For a positive trace-class operator $H_0$ independent of an isonormal row
$\xi$, $\mathcal Q_\xi(H_0)$ denotes the Gaussian quadratic form (finite
compression limit, in $L^1$ when $\E\tr H_0<\infty$).

\paragraph{The covariance observation.}
For finite input and output and a skew operator $B$, write $\partial_C$ for directional differentiation in
$M=M^\top$. Define
\begin{equation}\label{eq:generator}
 \begin{aligned}
 F_B(\lambda,M)&=\det(I+\lambda BM),\\
 \Phi_\lambda(\boldsymbol C;\boldsymbol t)
 &=\left[\prod_i(1+t_i\partial_{C_i})F_B(\lambda,M)\right]_{M=0}.
 \end{aligned}
\end{equation}
Write $\Phi_\lambda(\boldsymbol C)=\Phi_\lambda(\boldsymbol C;\boldsymbol q)$.
Here $\partial_CF_B=D_MF_B[C]$ differentiates the common matrix $M$;
$\nabla_{C_i}\Phi$ instead differentiates the $i$th covariance slot,
using the trace pairing.

\begin{theorem}[Gaussian conditional spectra]\label{thm:main}
\textup{(a) Observation, conditional spectrum, and response.} For every integer $r\ge1$, exterior evaluation extends uniquely
to a continuous linear map
$\Gamma_r:\Lambda^rK\to L^2(\Omega;\Lambda^rH)$ satisfying
$\E\|\Gamma_ru\|^2=r!e_r(q)\|u\|^2$.
Fix $\eta\in\Lambda^2K$ with skew operator $B$, and define $A$ by $\omega_A=\Gamma_2\eta$.
Then $A\in L^2(\Omega;\Sp_2)$ and
\begin{equation}\label{eq:mean}
 P(\lambda):=\E Z_A(\lambda)
 =\sum_{k\ge0}(2k)!e_{2k}(q)
       \left\|\frac{\eta^{\wedge k}}{k!}\right\|^2\lambda^{2k},\qquad \lambda\ge0.
\end{equation}
The function $P$ is entire, and a deterministic skew-adjoint $T\in\Sp_2(H)$ satisfies
\begin{equation}\label{eq:meanmodel}
 Z_T=P,\qquad
 \|T\|_2^2=4e_2(q)\|\eta\|^2=\E\|A\|_2^2.
\end{equation}
The nonzero spectrum of $T$, with multiplicities, is uniquely determined by $q$
and the input singular-value pairs. Almost surely, simultaneously for all
finite $J$, $P_J=\E[Z_A\mid\mathcal F_J]=Z_{T_J}$ for skew-adjoint
$T_J\in\Sp_2(H)$.
The models before and after revealing $i\notin J$ interlace a common
row-deleted spectrum; their counts differ by at most two at every positive
threshold.
In finite dimensions these are evaluations of one observation:
\[
 P_J=\Phi_\lambda(\boldsymbol C^J),\qquad
 C_i^J=\begin{cases}\xi_i\xi_i^\top,&i\in J,\\I,&i\notin J.\end{cases}
\]
For $\lambda\ge0$ and $i\notin J$, its normalized covariance gradients
$H_{i\mid J}=\nabla_{C_i}\log\Phi_\lambda(\boldsymbol C^J)$, using the trace
pairing, extend to positive trace-class operators in Hilbert space.
They are $\mathcal F_J$-measurable, satisfy $\tr H_{i\mid J}<1$, and obey
\begin{equation}\label{eq:response}
 \frac{P_{J\cup\{i\}}}{P_J}=1-\tr H_{i\mid J}+\mathcal Q_{\xi_i}(H_{i\mid J}),
 \qquad 0\le\sum_{i\notin J} H_{i\mid J}\le I_K.
\end{equation}
The sum converges in trace norm.

\smallskip\noindent\textup{(b) Uniform spectral comparison.}
Let $\mathfrak J$ be the finite stages of the chosen rule together with
$*$. The least pathwise dilation (with $\inf\varnothing=\infty$)
\begin{equation}\label{eq:dilation}
 L:=\inf\{\ell\ge1:P(\lambda/\ell)\le Z_{T_j}(\lambda)\le P(\ell\lambda)
                       \text{ for all }j\in\mathfrak J,\lambda\ge0\}
\end{equation}
is measurable and satisfies the universal bound
\begin{equation}\label{eq:weakone}
 \Pbb(L>\ell)\le\frac5\ell,\qquad \ell\ge1.
\end{equation}
On one event $\Omega_0$ of probability one, $L<\infty$ attains its infimum and
\begin{equation}\label{eq:scale}
 L^{-1}\|T\|_p\le\|T_j\|_p\le L\|T\|_p
 \qquad\text{for all }j\in\mathfrak J,\ 0<p<2,
\end{equation}
with infinite values allowed, and $T_j\in\Sp_p$ for all $j\in\mathfrak J$, $p\ge2$.
The power $1$ in \eqref{eq:weakone} is optimal among universal tail bounds.
For $0<p<2$ the extended expectations satisfy
\begin{equation}\label{eq:mean-schatten}
 \frac{5^{-p}}{p+1}\|T\|_p^p\le \E\|T_j\|_p^p\le\|T\|_p^p,
 \qquad j\in\mathfrak J.
\end{equation}

If $T\ne0$, on the same event,
\begin{equation}\label{eq:resolvent-asymptotic}
 \log Z_{T_j}(\lambda)\sim\log P(\lambda),\qquad
 \mathcal R_{T_j}(\lambda)\sim\mathcal R_T(\lambda)
 \quad(\lambda\to\infty),
\end{equation}
uniformly in $j\in\mathfrak J$.

\smallskip\noindent\textup{(c) Exact concentration.}
Write $b_h$ for the squared positive singular-value pairs of the operator $B$
of $\eta$ (one value per pair). Put $D_\lambda=Z_A/P$,
$d=\lambda P'/(2P)$. With $P=\sum_k(2k)!e_{2k}(q)e_k(b)\lambda^{2k}$ viewed as a function of the
positive weights $q,b$, define
\begin{equation}\label{eq:two-matrices}
 u_{ah}=\frac{q_ab_h}{P}\partial_{q_a}\partial_{b_h}P,\quad
 v_{ab}=\frac{q_aq_b}{P}\partial_{q_a}\partial_{q_b}P\ (a<b),\quad
 U=\frac{(\sum_{a,h}u_{ah}^2)^{1/2}}{1+d},\quad
 V=\frac{(\sum_{a<b}v_{ab}^2)^{1/2}}{1+d}.
\end{equation}
Along any sequence of models and scales,
\begin{equation}\label{eq:general-concentration}
 D_\lambda\to1\text{ in probability}
 \ \Longleftrightarrow\ U+V\to0
 \ \Longleftrightarrow\ \sup_{j\in\mathfrak J}|P_j/P-1|\to0\text{ in every finite }L^p,\ p>0.
\end{equation}
They are also equivalent to $\mathcal V(\lambda)=\sum_j\tr(H_{i_j\mid J_{j-1}}(\lambda)^2)\to0$
in probability. If $d\to\infty$, these conditions are equivalent to
$\sum_a(q_a\partial_{q_a}\log P)^2/d\to0$.

For fixed infinite-rank $Q,B$ as $\lambda\to\infty$, these are equivalent to
$\|\Pi_4D_\lambda\|_2\to0$ and to $tN_Q(t)/\tau_Q(t)\to0$ as $t\downarrow0$
(equivalently, $\tau_Q$ is slowly varying at zero), where $\Pi_r$ denotes
projection onto the $r$th homogeneous Wiener chaos.
\end{theorem}

Conditioning on finite index histories gives
$P_{J_j}=\E[Z_A\mid\mathcal G_j]$; pathwise claims use a
rule-dependent probability-one event.

\section{The covariance observation}\label{sec:randomization}
\subsection{Exterior evaluation}
The independent isonormal rows $\xi_i$ \cite{Janson1997} give
$Gf=\sum_i\sqrt{q_i}\xi_i(f)e_i$ in $L^2$.
For $w\in\Lambda^dK$, $d\ge1$, set $T_wv=\iota_vw$, with contraction adjoint to
left exterior multiplication. Counting each wedge coefficient $d$ times gives
$\tr(T_w^*T_w)=d\|w\|^2$. For an independent row, the series
$\iota_\xi w=\sum_a\xi(f_a)\iota_{f_a}w$ converges in
$L^2(\Omega;\Lambda^{d-1}K)$, is basis-independent, and has covariance $T_wT_w^*$.
For increasing $I=(i_1,\ldots,i_m)$, write
$\iota_{\xi_I}=\iota_{\xi_{i_m}}\cdots\iota_{\xi_{i_1}}$ and
$e_I=e_{i_1}\wedge\cdots\wedge e_{i_m}$. For finite $J$, contract
$I\cap J$ first, up to sign, and integrate the independent remaining rows:
\begin{equation}\label{eq:scalarsecond}
 \E[|\iota_{\xi_I}u|^2\mid\mathcal F_J]
 =(m-|I\cap J|)!\|\iota_{\xi_{I\cap J}}u\|^2.
\end{equation}
Hence exterior evaluation extends by
\begin{equation}\label{eq:exterior-evaluation}
 \Gamma_mu=\sum_{|I|=m}\sqrt{q_I}(\iota_{\xi_I}u)e_I,\qquad
 \E\|\Gamma_mu\|^2=m!e_m(q)\|u\|^2.
\end{equation}
Set $\Gamma_0=\operatorname{id}_{\mathbb R}$. Product preservation on
finite inputs and continuity of exterior multiplication in probability give,
simultaneously in $k$,
\begin{equation}\label{eq:powers}
 u_k=\eta^{\wedge k}/k!,\qquad
 \omega_A^{\wedge k}/k!=\Gamma_{2k}u_k,\qquad
 \E\|\Gamma_{2k}u_k\|^2=(2k)!e_{2k}(q)\|u_k\|^2.
\end{equation}
Let $B$ be the skew operator of $\eta$, with singular-value pairs $s_h$.
Since $\|u_k\|^2=e_k(s^2)$ and $(2k)!e_{2k}(q)\le S^{2k}$,
\begin{equation}\label{eq:entire-majorant}
 \sum_k\E\|\Gamma_{2k}u_k\|^2|z|^k
 \le\prod_h(1+S^2s_h^2|z|).
\end{equation}
Tonelli proves \eqref{eq:mean}; this entire majorant and its derivative
control the energy and coefficient-weighted energy sums.

\subsection{One covariance observation}

We first work in finite input and output dimensions.
Sylvester and \eqref{eq:det} give
\begin{equation}\label{eq:covariance-kernel}
 F_B(\lambda,M)=\sum_k\lambda^{2k}\langle u_k,\Lambda^{2k}M\,u_k\rangle,
 \qquad \partial_{vv^\top}^2F_B=0.
\end{equation}
The second identity is the rank-one determinant formula. Each coefficient
is nonnegative on $M\ge0$; its PSD directional derivatives share this
positivity and rank-one affinity, by nonnegative slopes and commutation.
For independent $x$ with $\E xx^\top=qC$, rank-one affinity gives
$\E f(M+xx^\top)=(1+q\partial_C)f(M)$.
Thus $\Phi_\lambda(\boldsymbol C)=\E\det(I+\lambda XBX^\top)$ for
independent rows with second moments $q_iC_i$. Taking
$C_i^J=\xi_i\xi_i^\top$ on $J$ and $C_i^J=I$ otherwise gives $P_J$.

Restricting $B$ to its support gives
$F_B(i,M)=\det(iB)\det(M+(iB)^{-1})\ne0$ for $\operatorname{Im}M>0$;
$B=0$ is trivial. This nonvanishing property, allowing other
upper-half-plane parameters, is preserved by $f+t\partial_Cf$ for
polynomial $f$, $C\ge0$ and $\operatorname{Im}t>0$.
Indeed the roots $\alpha$ of $f(M+wC)$ lie below the axis because
$\operatorname{Im}M-\epsilon C>0$ for small $\epsilon>0$.
Unless constant, it gives
$\operatorname{Im}(\partial_Cf/f)=\operatorname{Im}\sum_\alpha(-\alpha)^{-1}<0$,
so the possible zero $t=-f/\partial_Cf$ is below the axis.
Iterate and let $M=i\epsilon I\to0$: Hurwitz and $\Phi_i(\boldsymbol C;0)=1$
give real stability. As $-q_i/x_i$ preserves the upper half-plane,
\[
 p_J(\boldsymbol x)=\Bigl(\prod_i x_i\Bigr)
 \Phi_i(\boldsymbol C^J;(-q_i/x_i)_i)
 =\E[\det(D_x-iA)\mid\mathcal F_J].
\]
This polynomial is real stable. Its monic diagonal has real roots in opposite pairs,
realizing $P_J$. Multiaffinity makes $\partial_{x_i}p_J$ the deletion,
independent of $C_i$; its diagonal $g$ interlaces since, for every real $c$,
$p_J+cg$ is the substitution $x_i=x+c$. Both stages share this deletion, with
singular-value counts differing by at most two.

For later use, still in finite dimensions, let $f,g$ be $\Phi$ at fixed
real scales, or scalar multiples of its fixed-degree coefficients.
Write $\widehat f=f((\xi_i\xi_i^\top)_i)$, and let
$D_Ef(\boldsymbol I)\in\bigotimes_{i\in E}\operatorname{Sym}(K)$ be
its derivative once in each row variable of $E$, represented by the
trace inner product in each slot. With $W_i=\xi_i\xi_i^\top-I$,
separate affinity and Gaussian independence give
\begin{equation}\label{eq:covariance-chaos}
 \begin{aligned}
 \widehat f&=\sum_E\left\langle D_Ef(\boldsymbol I),
                         \bigotimes_{i\in E}W_i\right\rangle,\\
 \langle\Pi_{2r}\widehat f,\Pi_{2r}\widehat g\rangle
 &=2^r\sum_{|E|=r}
    \langle D_Ef(\boldsymbol I),D_Eg(\boldsymbol I)\rangle .
 \end{aligned}
\end{equation}
Wick's identity $\E\langle H,W_i\rangle\langle L,W_i\rangle=2\tr(HL)$
and independence prove the second formula; each $E$-term has chaos order $2|E|$.

\subsection{Hilbert-space responses and spectral limit}
For even $I$, put $m=|I|$, $d=|I\setminus J|$ and
$w=\iota_{\xi_{I\cap J}}u_{m/2}$. Equation \eqref{eq:scalarsecond} identifies
the energy and its covariance gradient as
\begin{equation}\label{eq:residual-response}
 \begin{aligned}
 b_{I,J}&=q_I\E[|\iota_{\xi_I}u_{m/2}|^2\mid\mathcal F_J]
          =q_Id!\|w\|^2,\\
 C_{I,i\mid J}&=q_I(d-1)!T_w^*T_w\quad(d>0,\ i\in I\setminus J).
 \end{aligned}
\end{equation}
These define trace-class responses in arbitrary input dimension, with
$b_{I,J+i}=\mathcal Q_{\xi_i}(C_{I,i\mid J})$.
For $d>0$ and unit $v$, the orthogonal splitting
$w=v\wedge\iota_vw+\Lambda^dP_{v^\perp}w$ gives
$\|\iota_vw\|\le\|w\|$. Together with the trace identity this proves
\begin{equation}\label{eq:row-budget}
 \tr C_{I,i\mid J}=b_{I,J},\qquad
 dC_{I,i\mid J}\le b_{I,J}I_K.
\end{equation}
The same calculation holds for arbitrary alternating input.
With $b_{\varnothing,J}=1$, sum the energies and gradients:
\[
 P_J=\sum_I\lambda^{|I|}b_{I,J},\qquad
 P_JH_{i\mid J}=\sum_{I\ni i}\lambda^{|I|}C_{I,i\mid J}.
\]
At each integer radius, conditional Tonelli and the derivative of
\eqref{eq:entire-majorant} make the energy and total trace series finite.
Positivity gives local uniform convergence and, for countably many $J$,
common versions at all scales. Summing \eqref{eq:row-budget} gives
$\tr H_{i\mid J}\le1-P_J^{-1}<1$ and $\sum_{i\notin J}H_{i\mid J}\le I_K$.
The coefficient updates give the ratio in \eqref{eq:response}.

Choose nested finite-rank orthogonal $R_n\uparrow I_K$ and put
$u_{k,n}=\Lambda^{2k}R_nu_k$, corresponding to $\eta_n=\Lambda^2R_n\eta$.
For each finite even $I$ and finite $J$, the residuals satisfy
\[
 \E\|\iota_{\xi_{I\cap J}}(u_{m/2,n}-u_{m/2})\|^2
 =\frac{m!}{d!}\|u_{m/2,n}-u_{m/2}\|^2\longrightarrow0.
\]
A deterministic diagonal subsequence gives simultaneous almost-sure
residual convergence for all $I,J$, hence energy convergence and trace-norm
convergence of their gradients. At fixed finite output, the conditional
characteristic polynomials therefore converge coefficientwise. For
$\eta_n$, conditioning on full revealed rows equals conditioning on
their projections onto $R_nK$, so the finite-dimensional proof applies.
Hurwitz preserves real stability, since the leading monomial stays fixed.
If $i\notin I$, then $b_{I,J+i}=b_{I,J}$; thus the limiting polynomials
before and after revelation share their deletion and its interlacing.

Now let $P_N$ project onto increasing finite coordinate spaces containing
$\operatorname{span}\{e_i:i\in J\}$ and exhausting $\overline{\operatorname{ran}Q}$.
Their positive roots
$\sigma_{N,l}$, padded with zeros, increase coordinatewise by interlacing, and
\[
 \sum_l\sigma_{N,l}^2
 =\tfrac12\E[\|P_NAP_N\|_2^2\mid\mathcal F_J]
 \uparrow\tfrac12\E[\|A\|_2^2\mid\mathcal F_J].
\]
Monotone convergence gives
$\sum_l\sigma_l^2=\lim_N\sum_l\sigma_{N,l}^2$ for
$\sigma_l=\lim_N\sigma_{N,l}$, so no spectral mass escapes.
Conditional Tonelli and monotone convergence give
$\prod_l(1+\lambda^2\sigma_l^2)=P_J$.
Skew blocks $\sigma_l$ define $T_J$; ordered roots give measurability,
and the zeros determine its nonzero spectrum. Common output limits
retain deletion interlacing, proving \eqref{eq:meanmodel} and realization.

\section{Uniform spectral comparison}\label{sec:comparison}
\subsection{Convex order from covariance gradients}
For $u\ne0$ and $e_m(q)>0$, write
\[
 U_I(u)=\frac{|\iota_{\xi_I}u|^2}{m!\|u\|^2},\qquad
 \mathsf Y_m(u)=\sum_{|I|=m}\frac{q_I}{e_m(q)}U_I(u).
\]
Use $(m)_j=m!/(m-j)!$ and independent $G_d\sim\chi_d^2/d$, so
$\E G_d=1$ and $\E G_d^2=1+2/d$. For $J\subseteq I$,
\begin{equation}\label{eq:master-conditional}
 \E[U_I(u)\mid\mathcal F_J]
 \preceq_{\rm cx}\prod_{d=m-|J|+1}^mG_d .
\end{equation}
Indeed \eqref{eq:row-budget} gives normalized row eigenweights
$\sum_a c_a=1$, $0\le c_a\le1/d$ at remaining degree $d$.
Their finite polytope has extreme points with $d$ entries $1/d$;
convexity and exchangeability give $\sum_a c_a\chi_{1,a}^2\preceq_{\rm cx}G_d$.
Merge a tail of mass at most $1/d$ and pass in $L^1$ through stop-loss
functions for infinitely many weights. Iterate using convexity of
$x\mapsto\E\phi(xG_d)$; zero residuals stay zero. Decomposable inputs
give scaled rank-$d$ projections, hence equality. Jensen extends
$W_m=\prod_{d=1}^mG_d$ to $\mathsf Y_m(u)$.

\subsection{Coefficient bounds and a common dilation}
Write $\beta_k=(2k)!e_{2k}(q)\|u_k\|^2$ and $Y_k=\mathsf Y_{2k}(u_k)$
for $\beta_k>0$; otherwise put $Y_k=1$, including $k=0$.
Let $Y_{k,j}=\E[Y_k\mid\mathcal G_j]$, with terminal $Y_k$.
Then $Z_{T_j}(\lambda)=\sum_k\beta_kY_{k,j}\lambda^{2k}$ by conditional Tonelli.
For $d\ge2$, the gamma recurrence gives
$\E G_d^{-1/2}=\frac d{d-1}\E\sqrt{G_d}\le d/(d-1)$.
For $k\ge1$, conditioning on $G_2,\ldots,G_{2k}$ and using the Gaussian density gives
$\E(a-W_{2k})_+\le(4k/3)\sqrt{2/\pi}\,a^{3/2}$.
Convex order gives the same bound for $Y_k$.
Conditional Jensen and Doob's inequality, with $a=3x$, give for $x,b>0$
\begin{equation}\label{eq:gaussian}
 \Pbb(\exists j:Y_{k,j}\le x)\le2\sqrt{6/\pi}\,k\sqrt x,\qquad
 \Pbb(\exists j:Y_{k,j}>b)\le b^{-1}.
\end{equation}
Continuity and rational tests make $L$ measurable and its finite infimum
attained. Bounds $\ell^{-2k}\le Y_{k,j}\le\ell^{2k}$ for all $k,j$ imply
\eqref{eq:dilation}; their union bound gives
\begin{equation}\label{eq:dilationtail}
 \Pbb(L>\ell)\le
 \sum_{k\ge1}\bigl(2\sqrt{6/\pi}\,k\ell^{-k}+\ell^{-2k}\bigr)
 =\frac{1}{\ell^2-1}+2\sqrt{6/\pi}\frac{\ell}{(\ell-1)^2}
 \quad(\ell>1).
\end{equation}
For $\ell\ge5$ the right side times $\ell$ is bounded by
$5/24+(25/8)\sqrt{6/\pi}<5$; for smaller $\ell$ use $\Pbb(L>\ell)\le1$.
This proves \eqref{eq:weakone} (including the trivial constant determinant).

The same coefficient tails give one finite random $M\ge1$ such that
\begin{equation}\label{eq:polynomial-envelope}
 M^{-2}(k+1)^{-6}\le Y_{k,j}\le M(k+1)^2\qquad(k\ge0,\ \text{all }j).
\end{equation}
For deterministic $M$ the sum of failure probabilities is $O(M^{-1})$.
For integers $s,k\ge1$, gamma comparison yields
\begin{equation}\label{eq:coefficient-hypercontractivity}
 \E Y_k^s\le\prod_{d=1}^{2k}\prod_{j=0}^{s-1}(1+2j/d)
 \le e^{s(s-1)}(2k)^{s(s-1)}.
\end{equation}
\subsection{Schatten comparison and sharpness}
For skew-adjoint $T\in\Sp_2$, Tonelli and the scalar beta integral give
\begin{equation}\label{eq:integral}
 \|T\|_p^p=\frac{2p\sin(\pi p/2)}\pi
       \int_0^\infty\log Z_T(\lambda)\,\frac{d\lambda}{\lambda^{p+1}},
 \qquad 0<p<2,
\end{equation}
with infinite values allowed and each pair counted twice.
Integrating the logarithms in \eqref{eq:dilation} proves \eqref{eq:scale} for every $p$;
$\Sp_2\subset\Sp_p$ for $p\ge2$. Jensen gives the upper bound in
\eqref{eq:mean-schatten}; the lower follows from
\[
 \E L^{-p}=p\int_1^\infty\Pbb(L\le x)\,\frac{dx}{x^{p+1}}
 \ge p\int_5^\infty(1-5/x)\,\frac{dx}{x^{p+1}}=\frac{5^{-p}}{p+1}.
\]

\emph{Sharpness.} For $H=K=\mathbb R^2$, $Q=I$, $\eta=f_1\wedge f_2$
and $\Delta$ the standard $2\times2$ Gaussian determinant,
\[
 Z_A(\lambda)=1+\Delta^2\lambda^2,\qquad P(\lambda)=1+2\lambda^2,
 \qquad L\ge\sqrt2/|\Delta|.
\]
An orthogonal coordinate change writes $\Delta$ as the difference of
independent mean-one exponentials, with density $e^{-|x|}/2$.
Thus $\liminf_{\ell\to\infty}\ell\Pbb(L>\ell)\ge\sqrt2$,
proving that the exponent in \eqref{eq:weakone} is optimal.

\subsection{Resolvent asymptotics}
For the positive singular-value pairs $\sigma_l(T)$, put
$p_l=\lambda^2\sigma_l(T)^2/(1+\lambda^2\sigma_l(T)^2)$.
Let $K_\lambda$ be a sum of independent Bernoulli variables with these
parameters, and write $\E_B$ for auxiliary expectation. Its generating
function is $P(\lambda\sqrt z)/P(\lambda)$, so
\begin{equation}\label{eq:positive-determinant-mixture}
 \frac{Z_A(\lambda)}{P(\lambda)}=\E_B Y_{K_\lambda},\qquad
 \Pbb_B(K_\lambda=k)=\frac{\beta_k\lambda^{2k}}{P(\lambda)},\qquad
 \frac{P_j}{P}=\E_B Y_{K_\lambda,j}.
\end{equation}
Write $K=K_\lambda$, so $\E_BK=d$ and $\operatorname{Var}_BK\le d$.
Bernoulli factorial moments give $\E_B(1+K)^s\le C_s(1+d)^s$;
Chernoff and Cauchy--Schwarz give
\begin{equation}\label{eq:bernoulli-tail}
 \E_B[(1+K)^s;|K-d|>\varepsilon d]
 \le C_s(1+d)^s e^{-\varepsilon^2d/6},
 \qquad s\ge0,\quad0<\varepsilon<1.
\end{equation}
The moment bound and \eqref{eq:coefficient-hypercontractivity} give
convergence of \eqref{eq:positive-determinant-mixture} in every finite $L^p$.
Averaging \eqref{eq:polynomial-envelope} gives
$M^{-2}(1+d)^{-6}\le P_j/P\le2M(1+d)^2$.
Hence $|\log P_j-\log P|\le C(\omega)\log(2+d)$ uniformly in $j$,
which is $o(\log P)$ since $\log P\ge d$ and $\log P\to\infty$ for $T\ne0$.

Coefficient differentiation gives
$\mathcal R_{T_j}/2=\E_B[K Y_{K,j}]/\E_BY_{K,j}$.
For $d\to\infty$, splitting at $|K-d|\le\varepsilon d$ and using
\eqref{eq:bernoulli-tail} and the envelope bounds
$|\mathcal R_{T_j}/(2d)-1|$ by
$\varepsilon+O_\omega((1+d)^8e^{-\varepsilon^2d/6})$, uniformly in $j$;
let $d\to\infty$, then $\varepsilon\downarrow0$.
For positive finite degree $r$, $K_\lambda\to r$ in probability; the envelope
is bounded above and below on $0\le k\le r$, so this mean tends uniformly to $r$.
This proves \eqref{eq:resolvent-asymptotic}.

\subsection{From resolvents to counts}
Layer cake gives
\begin{equation}\label{eq:resolvent-from-count}
 \mathcal R_T(1/t)=\int_0^\infty\frac{2u}{(1+u^2)^2}N_T(tu)\,du.
\end{equation}
The following elementary Stieltjes-transform argument will give exact
Wiener-area counts without a regular variation assumption; for related
Tauberian results see \cite[Theorem~A.7 and Remark~A.8]{LangerWoracek2025}.

\begin{lemma}[Counting transfer]\label{lem:count-transfer}
Suppose $T$ has infinite rank, $\mathcal R_T(1/t)=O(N_T(t))$, and
\begin{equation}\label{eq:threshold-band}
 \lim_{\varepsilon\downarrow0}\limsup_{t\downarrow0}
 \frac{N_T((1-\varepsilon)t)-N_T((1+\varepsilon)t)}{N_T(t)}=0.
\end{equation}
Then $N_{T_j}(t)\sim N_T(t)$ uniformly in the conditional stages.
\end{lemma}
\begin{proof}
Write $d_t=N_T(t)\to\infty$.
For $C=T_j,T$, set $y_i=s_i(C)^2/(t^2+s_i(C)^2)$ and
$\mu_{C,t}=d_t^{-1}\sum_i y_i\delta_{y_i}$ on $[0,1]$.
By \eqref{eq:resolvent-asymptotic}, these measures have uniformly bounded
masses. Their integrals of $a/[1+(a-1)y]$ equal
$\mathcal R_C(\sqrt a/t)/d_t$ and differ by $o(1)$ uniformly in $j$,
since $\mathcal R_T(\sqrt a/t)\le\max(1,a)\mathcal R_T(1/t)$.
Finite differences at $a=1$ generate polynomial tests; uniform approximation
extends this convergence to all continuous tests. Continuous upper and
lower approximations to $y^{-1}\mathbf1_{\{y>1/2\}}$, with
\eqref{eq:threshold-band}, give the claimed uniform counting limit.
\end{proof}

\section{Concentration through overlap}\label{sec:concentration}
\subsection{Row retention and low-order overlap}
The convex order \eqref{eq:master-conditional} controls all chaoses.
The symmetric row decomposition
$U_I(u)=\sum_{E\subseteq I}h_{m,|E|}(\xi_E)$ has orthogonal kernels of
chaos order $2|E|$, since each row enters quadratically. Put
$a_{m,r}=\binom mr\|h_{m,r}\|_2^2=\|\Pi_{2r}U_I(u)\|_2^2$, with
$a_{m,0}=1$. Retain rows independently with probability $z$, forming
$\mathcal E$ with $J=|\mathcal E|\sim\operatorname{Bin}(m,z)$. Conditioning
keeps exactly the kernels supported in $\mathcal E$; hence
\begin{equation}\label{eq:row-retention}
 \begin{aligned}
 F_m(z):=\sum_{r=0}^m a_{m,r}z^r
 &=\E_{\mathcal E}\|\E[U_I(u)\mid\mathcal F_{\mathcal E}]\|_2^2\\
 &\le\E\frac{(m+1)(m+2)}{(m-J+1)(m-J+2)}
 =\sum_{r=0}^m(r+1)z^r,\qquad 0\le z\le1.
 \end{aligned}
\end{equation}
Here \eqref{eq:master-conditional} gives the inequality; completing the
$\operatorname{Bin}(m+2,z)$ sum gives the last equality. In particular
$F_m(z)\le(1-z)^{-2}$ for $z<1$ and $F_m(1)-1\le m(m+3)/2$.
\begin{lemma}[Input--output overlap]\label{lemma:replica-overlap}
Write $\eta=\sum_hs_hp_h$ in oriented orthonormal skew planes,
with $s_h>0$ and $b_h=s_h^2$. For $k,l\ge1$ with $\beta_k\beta_l>0$, put
$m=2k$, $n=2l$. Independently sample output sets $I,I'$ of sizes $m,n$
with product weights $q$, and input plane sets $S,T$ of sizes $k,l$
with product weights $b$. Write
$B_{m,n}^{(r)}=\E\binom{|I\cap I'|}{r}$, $B_{kl}=B_{m,n}^{(1)}$,
$J_{kl}=B_{m,n}^{(2)}$ and $X=|S\cap T|$. Then
\begin{equation}\label{eq:full-fourth}
 \begin{split}
 \langle\Pi_2Y_k,\Pi_2Y_l\rangle&=\frac{4B_{kl}\E X}{mn},\\
 \langle\Pi_4Y_k,\Pi_4Y_l\rangle
 &=\frac{J_{kl}\E[12X+48\binom X2+8(k-X)(l-X)]}{(m)_2(n)_2}
 \ge\frac{6J_{kl}}{5\max(m,n)^2}.
 \end{split}
\end{equation}
\end{lemma}
\begin{proof}
For $r=1,2$ in finite dimensions, put
$G_k(A,B)=\langle\iota_Au_k,\iota_Bu_k\rangle/e_k(b)$ for increasing
$r$-coordinate sets. Embed this Gram isometrically into $K^{\otimes2r}$,
and put $\mathsf S_r=\prod_{a=1}^r[1+(a,r+a)]/2$.
The normalized derivative in rows $E$, $|E|=r$, is
$\Pbb(E\subseteq I)\mathsf S_rG_k/\binom mr$: the ordered contraction
Gram is $r!G_k$ and homogeneity divides by $(m)_r$.
Thus \eqref{eq:covariance-chaos} gives
\begin{equation}\label{eq:all-row-gram}
 \langle\Pi_{2r}Y_k,\Pi_{2r}Y_l\rangle
 =\frac{2^rB_{m,n}^{(r)}}{\binom mr\binom nr}
   \langle G_k,\mathsf S_rG_l\rangle,\qquad r=1,2.
\end{equation}
For $r=1$, $G_k$ is diagonal with each input-plane inclusion probability
repeated twice, so its mixed inner product is $2\E X$.
For $r=2$, nonzero off-diagonal entries move one whole plane to another.
The diagonal and off-diagonal mixed masses are respectively
$D=\E\binom{2X}2$ and $O=\E[(k-X)(l-X)]$: the first counts common
coordinate pairs, while the second chooses one exclusive plane on each
side. In the four terms of $4\mathsf S_2$, the identity and full exchange
each contribute $D+O$; each single exchange contributes $D/2$ by normalized
wedges, while whole-plane moves split and vanish. Thus $4\langle G_k,\mathsf S_2G_l\rangle=3D+2O$,
which proves the equalities. If $k\le l$, the numerator is at least
$12X^2+8(k-X)^2\ge24k^2/5$, proving the lower bound.

For input and output truncations, \eqref{eq:master-conditional} gives
$L^4$ continuity of exterior evaluation and $L^2$ convergence of squared
energies. Fixed-size sample laws converge in total variation and the
low-order overlap counts are bounded, proving the Hilbert-space formulas.
\end{proof}
\subsection{Weighted sampling and degree conditioning}
For $m\ge1$ with $e_m(q)>0$, sample $I_m$ with probability $q_I/e_m(q)$,
and put $\pi_i^{(m)}=\Pbb(i\in I_m)$ and
$R_m=m^{-1}\sum_i(\pi_i^{(m)})^2$. Conditioning on inclusion deletes those
weights and reduces the size. Newton's inequality
$e_{r-1}^2\ge e_{r-2}e_r$ on the remaining weights gives negative
correlation ($e_{-1}=0$)
\cite{BorceaBrandenLiggett2009}; iteration yields
$\Pbb(E\subseteq I_m)\le\prod_{i\in E}\pi_i^{(m)}$.
The odds $q_ie_{m-1}(q_{\widehat i})/e_m(q_{\widehat i})$ increase with
$m$ by the same inequality, as do fixed-subset inclusion probabilities
by successive conditioning. Finite-truncation total-variation convergence
extends these facts to summable weights.

Let $K=K_\lambda$ be the Bernoulli degree in
\eqref{eq:positive-determinant-mixture}. Given $K=k$, draw independent product-weight sets $I,S$ of sizes $2k,k$
with weights $q,b$. Coefficient differentiation gives
$u_{ah}=\Pbb(a\in I,h\in S)$ and $v_{ab}=\Pbb(a,b\in I)$.
For either positive-probability event $F$,
\begin{equation}\label{eq:degree-comparison}
 \max(1,K)\preceq_{\rm st}(K\mid F)\preceq_{\rm st}K+2,
 \qquad \frac1{d+2}\le\E[K^{-1}\mid F]\le\frac4{1+d}.
\end{equation}
Increasing inclusion probabilities give the lower stochastic bound.
In finite rank, $P$, as a polynomial in $z=\lambda^2$, has nonnegative
coefficients and nonpositive zeros for all positive parameters. Its
parameter derivatives inherit this by multi-affinity and coefficient
limits. For $R=P$ or a nonzero selected derivative and a remaining weight
$x$ with $\partial_xR\ne0$, $R$ and $R+t\partial_xR$, $t>0$, have the same degree
and a common interlacer \cite[Lemma~3.5]{MSS2015}, since their positive
combinations are parameter shifts. Let $K_R$ have generating function
$R(zu)/R(z)$. As $t\to\infty$, their decreasing Bernoulli parameters give
$p'_{j+1}\le p_j$ for $\partial_xR$ versus $R$; zero roots give parameter
one and lost degrees are padded by zeros. Common uniforms yield
$K_{\partial_xR}\preceq_{\rm st}K_R+1$. A single output inclusion therefore
gives $(K\mid i\in I)\preceq_{\rm st}K+1$; two derivatives give the
upper bound for $F$. Finite truncation containing the specified indices
passes these bounds by total-variation convergence.
Jensen gives the lower inverse-moment bound; the upper follows from
$\E(1+K)^{-1}=\int_0^1\E u^K\,du\le\min(1,1/d)$ and
$1/\max(1,k)\le2/(1+k)$.

\subsection{Detection and concentration}
We prove the estimates that detect concentration:
\begin{equation}\label{eq:matrix-detection}
 \begin{gathered}
 \|\Pi_2D_\lambda\|_2\asymp U,\qquad \|\Pi_4D_\lambda\|_2\asymp V,\qquad
 \|D_\lambda-1\|_1\ge c(U^2+V^2)^{1/2},\\
 \sum_{r\ge2}z^r\|\Pi_{2r}D_\lambda\|_2^2\le16(1-z)^{-2}V^2
 \quad(0<z<1).
 \end{gathered}
\end{equation}
Averaging the $r=1$ identity in \eqref{eq:full-fourth} gives
\[
 \|\Pi_2D_\lambda\|_2^2=\sum_{a,h}u_{ah}^2
 \bigl(\E[K^{-1}\mid a\in I,h\in S]\bigr)^2.
\]
Thus $U/2\le\|\Pi_2D_\lambda\|_2\le4U$.
For the fourth chaos, average the lower bound in \eqref{eq:full-fourth}
over independent degrees $K,L$ conditioned on each pair event.
Since $\E[\max(K,L)\mid a,b\in I,\ a,b\in I']\le2(d+2)$,
Jensen and \eqref{eq:degree-comparison} give
$\|\Pi_4D_\lambda\|_2^2\ge3V^2/160$.
The row kernels in \eqref{eq:row-retention} give, for all $r$,
$c_r:=\langle\Pi_{2r}Y_k,\Pi_{2r}Y_l\rangle
=B_{m,n}^{(r)}\langle h_{m,r},h_{n,r}\rangle$. Thus
\[
 |c_r|\le\frac{B_{m,n}^{(r)}}{\sqrt{\binom mr\binom nr}}
           \sqrt{a_{m,r}a_{n,r}}.
\]
For $2\le r\le m\le n$, counting extensions of a pair in $I\cap I'$ gives
$B_{m,n}^{(r)}\le J_{kl}\binom mr/\binom m2$. Consequently the fraction
above is at most $J_{kl}(m/n)^{r/2}/\binom m2\le4J_{kl}/(mn)$.
Cauchy--Schwarz and \eqref{eq:row-retention}, with
$F_m(1)-1\le5m^2/4$ for $m\ge2$, give
\begin{equation}\label{eq:mixed-higher-upper}
 \sum_{r\ge2}z^r|\langle\Pi_{2r}Y_k,\Pi_{2r}Y_l\rangle|
 \le J_{kl}\min\left\{5,\frac4{mn(1-z)^2}\right\},\qquad 0<z<1.
\end{equation}
After degree averaging, the second bound becomes
$(1-z)^{-2}\sum_{a<b}v_{ab}^2(\E[K^{-1}\mid a,b\in I])^2$.
This proves the bound $16V^2/(1-z)^2$ in
\eqref{eq:matrix-detection}, and $\|\Pi_4D_\lambda\|_2\le16V$
by taking $z=1/2$.

For $F=D_\lambda-1$, let $P_\rho$ multiply chaos of order $r$ by
$\rho^r$, and fix $\rho=1/4$. The lower bounds just proved give
$\|P_\rho F\|_2\gtrsim(U^2+V^2)^{1/2}$.
Hypercontractivity and the smoothed bound give
$\|P_\rho F\|_4\le\|P_{\sqrt3\rho}F\|_2\lesssim(U^2+V^2)^{1/2}$.
Interpolation and $L^1$ contractivity imply
$\|F\|_1\ge\|P_\rho F\|_2^3/\|P_\rho F\|_4^2$.
This proves the linear detection bound; $d=0$ is trivial.

Write $a_i(k)=\pi_i^{(2k)}$, $a(0)=0$, $\bar\pi=\E a(K)$ and
$A_0=\|\bar\pi\|_2^2$, so $\bar\pi_i=q_i\partial_{q_i}\log P$.
Set $R_0=0$.
Monotonicity and the total mass $2k$ give
$\|a(k)-a(l)\|_1=2|k-l|$ on admissible degrees, hence
\begin{equation}\label{eq:degree-marginals}
 |2kR_{2k}-A_0|\le2\|a(k)-\bar\pi\|_1
 \le4\E|k-K|\le4(|k-d|+\sqrt d).
\end{equation}
For independent $I,I'$, $Z=|I\cap I'|$ satisfies $\E Z=A_0$ and
$\E\binom Z2=(1+d)^2V^2$, so Jensen gives
\begin{equation}\label{eq:pair-marginal-bound}
 (A_0-1)_+\le\sqrt2(1+d)V.
\end{equation}

For $m=2k\ge2$ with $\beta_k>0$, the row decomposition gives
$\|\Pi_{2r}Y_k\|_2^2=B_{m,m}^{(r)}a_{m,r}/\binom mr$.
Negative correlation gives $B_{m,m}^{(r)}\le(mR_m)^r/r!$, so
$(m)_r\ge(m/e)^r$ bounds this by $a_{m,r}(eR_m)^r$.
Hypercontractivity and the summed bound \eqref{eq:row-retention} imply
\begin{equation}\label{eq:chaos-bound}
 \|Y_k-1\|_p^2\le(1-z)^{-2}-1,
 \qquad z=e(p-1)^2R_m<1,\quad p\ge2.
\end{equation}

Necessity in \eqref{eq:general-concentration} follows from
$\|D_\lambda-1\|_1=2\E(1-D_\lambda)_+$ and the linear estimate.
Suppose $U+V\to0$. For bounded $d$, \eqref{eq:mixed-higher-upper} gives
$\|D_\lambda-1\|_2^2\le16U^2+5(1+d)^2V^2\to0$.
Equation \eqref{eq:coefficient-hypercontractivity} and Bernoulli moments
give uniformly bounded higher moments, hence every finite $L^p$ convergence.

For $d\to\infty$, \eqref{eq:pair-marginal-bound} reduces the claim to
$A_0/d\to0$. Fix an integer $p\ge2$ and $0<\varepsilon<1/2$.
On $|k-d|\le\varepsilon d$, \eqref{eq:degree-marginals} gives
$R_{2k}\le[A_0/d+4\varepsilon+4/\sqrt d]/[2(1-\varepsilon)]$.
For $\varepsilon$ sufficiently small depending on $p$,
\eqref{eq:chaos-bound} bounds $\|Y_k-1\|_p$ by
$C_p\sqrt{A_0/d+\varepsilon+d^{-1/2}}$ for all large $d$.
On the complement use $\|Y_k-1\|_p\le C_p(1+k)^{p-1}$ from
\eqref{eq:coefficient-hypercontractivity} and \eqref{eq:bernoulli-tail}.
Minkowski gives $\limsup\|D_\lambda-1\|_p\le C_p\sqrt\varepsilon$;
let $\varepsilon\downarrow0$. Bounded and diverging subsequences exhaust
the possibilities; Doob's inequality gives the stage supremum.

\subsection{Fixed covariance and the tail criterion}
For infinite-rank $Q$, set $a_m=e_m(q)/e_{m-1}(q)$ and
$w_i(t)=\min(1,q_i/t)$. Newton's inequalities make $a_m$ decreasing
and $a_m\le S/m$ by counting extensions of an $(m-1)$-set.
With $E_j=e_j(q_{\widehat i})$, Newton also gives
$E_m/E_{m-1}\le a_m$ and $e_{m-1}(q)\ge E_{m-1}$. The inclusion
formula therefore gives
\begin{equation}\label{eq:overlap-marginals}
 \tfrac12 w_i(a_m)\le\frac{q_i}{q_i+a_m}\le\pi_i^{(m)}\le w_i(a_m).
\end{equation}
Put $F_2(t)=\sum_iw_i(t)^2$. Then $M_Q(a_m)\asymp m$ and
$F_2(a_m)\asymp mR_m$. On $a_{m+2}\le t\le a_m$, monotonicity gives
$F_2(t)/M_Q(t)\le4(m+2)R_{m+2}/m$, so both overlap limits are equivalent
to $F_2(t)/M_Q(t)\to0$. For $0<\varepsilon<1$,
\[
 N_Q(t)\le F_2(t)\le N_Q(\varepsilon t)+\varepsilon M_Q(t).
\]
With $r(t)=tN_Q(t)/\tau_Q(t)$ this gives
$r(t)/(1+r(t))\le F_2(t)/M_Q(t)\le r(\varepsilon t)/\varepsilon+\varepsilon$.
This gives the equivalence with $r(t)\to0$.
If $r(t)\to0$, then $0\le\tau_Q(ct)-\tau_Q(t)\le ctN_Q(t)$ for $c>1$,
so $\tau_Q$ is slowly varying. Conversely,
\[
 r(t/2)\le\frac{\tau_Q(t)}{2\tau_Q(t/2)}r(t)
          +\frac{\tau_Q(t)}{\tau_Q(t/2)}-1.
\]
Under slow variation the coefficient tends to $1/2$ and the additive
term to zero, giving $r(t)\to0$ on dyadic scales. For $t/2<s\le t$,
$r(s)\le2r(t/2)$, giving the full limit. In summary,
\begin{equation}\label{eq:overlap-tail-criterion}
 R_m\to0\quad\Longleftrightarrow\quad R_{2m}\to0
 \quad\Longleftrightarrow\quad\frac{tN_Q(t)}{\tau_Q(t)}\to0
 \quad\Longleftrightarrow\quad \tau_Q\ \text{is slowly varying at zero}.
\end{equation}

For fixed infinite ranks, $d(\lambda)\uparrow\infty$, so Jensen
$A_0\le2\E[K R_{2K}]$ shows that $R_{2k}\to0$ gives $A_0/d\to0$. Conversely,
\eqref{eq:matrix-detection} and \eqref{eq:pair-marginal-bound} give $A_0/d\to0$
whenever $\|\Pi_4D_\lambda\|_2\to0$. At $d(\lambda_n)=n$,
\eqref{eq:degree-marginals} gives $R_{2n}\le A_0/(2n)+2/\sqrt n\to0$.
This proves the fixed-data equivalences, including fourth-chaos detection.

\subsection{Predictable response}
Write $M_j=P_j/P$, $R_j=M_j/M_{j-1}$ and $v_j=\tr H_{i_j\mid J_{j-1}}^2$.
This positive mean-one martingale has terminal value $D_\lambda$.
Conditionally, $R_j-1$ has second moment $2v_j$ and fourth moment
$12v_j^2+48\tr H_{i_j\mid J_{j-1}}^4\le60v_j^2$. Interpolation and
Cauchy--Schwarz give, for $R=R_j$ with conditional expectation understood,
$\E|R-1|\ge\sqrt{2v_j/15}$ and $\E(\sqrt R+1)^2\le4$.
Since $1-\E\sqrt R=\frac12\E(\sqrt R-1)^2$, we obtain
$v_j/60\le1-\E\sqrt R\le v_j$. Summing the square-root martingale losses,
\[
 \tfrac1{60}\E\sum_j\sqrt{M_{j-1}}v_j\le1-\E\sqrt{D_\lambda}
 \le\E\sum_j\sqrt{M_{j-1}}v_j.
\]
Infinite sums follow by monotone convergence and $L^1$ martingale
convergence. If $D_\lambda\to1$ in probability along a sequence, it does
so in $L^1$ by positivity and mean one; Doob's inequality gives
$\sup_j|M_j-1|\to0$ in probability. The displayed bound proves $\mathcal V\to0$
on using $\inf_jM_j\ge1/2$ with probability tending to one.
Conversely stop just before the first update for which $M_{j-1}>2$ or
$\sum_{i\le j}v_i>\delta$. The stopped martingale has squared $L^2$ distance
from one at most $8\delta$, by summing its predictable variances.
On $\mathcal V\le\delta$ it either equals the original throughout or has reached
a value greater than two. The weak $L^2$ maximal inequality gives, for $0<\varepsilon<1$,
$\Pbb(|D_\lambda-1|>\varepsilon)\le\Pbb(\mathcal V>\delta)+8\delta/\varepsilon^2$.
Let $\delta\downarrow0$ after the limit. This proves the criterion without
any asymptotic assumptions on the individual models.

\section{Application to Wiener area}
\label{sec:covariance-area}\label{sec:projective-hilbert}

Let $W$ be an $H$-valued Wiener process starting at zero, with
trace-class covariance $Q\ge0$, and put $S=\tr Q$. Its skew area is
\begin{equation}\label{eq:area}
 A=\frac12\int_0^1(W_s\otimes dW_s-dW_s\otimes W_s).
\end{equation}
It\^o isometry \cite{Metivier1982} gives squared $L^2(\Sp_2)$ norm
$S^2/2$ for the first integral; the second is its adjoint.
Thus $A$ is skew-adjoint Hilbert--Schmidt. Antisymmetrization cancels
the It\^o--Stratonovich correction. Write $A_{s,t}$ for the area
of the centered increment path on $[s,t]$.

The projective norm
$\|u\|_\pi=\inf_{u=\sum_jx_j\otimes y_j}\sum_j\|x_j\|\,\|y_j\|$
identifies $H\widehat\otimes_\pi H$ isometrically with $\Sp_1(H)$.
For a deterministic partition $\mathcal D=\{0=t_0<\cdots<t_m=1\}$,
write $|\mathcal D|=\max_i(t_i-t_{i-1})$ and
$\Delta_iW=W_{t_i}-W_{t_{i-1}}$. Its polygonal interpolation
$W^{\mathcal D}$ has terminal skew area
\begin{equation}\label{eq:polygonal-area}
 A_{\mathcal D}(W)=\tfrac12\sum_{i<j}
 (\Delta_iW\otimes\Delta_jW-\Delta_jW\otimes\Delta_iW).
\end{equation}
Write $\mathbf W^{\mathcal D}$ for its canonical lift, with second level
$\mathbb W^{\mathcal D}_{s,t}
=\int_s^t(W^{\mathcal D}_u-W^{\mathcal D}_s)\otimes dW^{\mathcal D}_u$
and skew part $A^{\mathcal D}_{s,t}$. For paths starting at zero set
\begin{equation}\label{eq:rough-metric}
 \rho_\alpha(\mathbf X,\mathbf Y)
 =\|X-Y\|_\alpha+\|\mathbb X-\mathbb Y\|_{2\alpha},
 \qquad 1/3<\alpha<1/2,
\end{equation}
where $\|F\|_\gamma=\sup_{s<t}\|F_{s,t}\|/|t-s|^\gamma$
and level two uses the projective norm. Geometric means a limit
of smooth canonical lifts in this metric. Define
$\mathcal H(Q)=\sum_iq_i\log(S/q_i)$ over the positive eigenvalues,
with $\mathcal H(0)=0$.

\begin{theorem}[Wiener area and its projective lift]
\label{thm:terminal-projective}\label{cor:area}
\textup{(a) Spectrum and covariance criteria.}
In any ordered covariance eigenbasis define
\begin{equation}\label{eq:model}
 (T_Q)_{ij}=\tfrac12\sqrt{q_iq_j}\operatorname{sgn}(j-i),
 \qquad T_Q|_{\ker Q}=0.
\end{equation}
Then $T_Q$ is skew-adjoint Hilbert--Schmidt and
\begin{equation}\label{eq:D}
 D_Q(\lambda):=\E Z_A(\lambda)=Z_{T_Q}(\lambda)
 =\frac{\det(I+\lambda Q/2)+\det(I-\lambda Q/2)}2.
\end{equation}
Its nonzero singular values are independent of the ordered eigenbasis.
The conditional stages obtained by revealing spatial Wiener rows obey
Theorem~\ref{thm:main} with $T=T_Q$. Put
$\theta_Q(x)=\sum_i\arctan(xq_i)$. In infinite rank, on one event
of probability one,
\begin{equation}\label{eq:general-area-counts}
 N_{T_j}(t)\sim N_{T_Q}(t)\sim\tfrac2\pi\theta_Q((2t)^{-1})
 \asymp M_Q(t)\qquad(t\downarrow0),
\end{equation}
uniformly in all conditional stages, including $A$, with universal
comparison constants for sufficiently small $t$.
Simultaneously for every stage, $T_j\in\Sp_p$ for $p>1$, and
\begin{equation}\label{eq:classification}
 T_j\in\Sp_p\quad\Longleftrightarrow\quad
 \begin{cases}
  \displaystyle\sum_iq_i^p<\infty,&0<p<1,\\[3pt]
  \mathcal H(Q)<\infty,&p=1.
 \end{cases}
\end{equation}
In infinite rank,
\begin{equation}\label{eq:area-concentration}
 Z_A(\lambda)/D_Q(\lambda)\longrightarrow1\ \text{in probability}
 \quad\Longleftrightarrow\quad
 tN_Q(t)/\tau_Q(t)\longrightarrow0,
 \qquad \lambda\to\infty,\ t\downarrow0.
\end{equation}
Equivalently, $N_Q(t)=o(N_{T_Q}(t))$, or $N_Q(t)=o(N_A(t))$
almost surely. These conditions also give
$\sup_j|P_j/D_Q-1|\to0$ in every finite $L^p$.

\smallskip\noindent\textup{(b) Canonical projective lift.}
The following are equivalent:
\begin{enumerate}[label=\textup{(\roman*)},leftmargin=2.4em]
\item $\mathcal H(Q)<\infty$;
\item $A\in\Sp_1$ almost surely;
\item $A_{\mathcal D}(W)$ has tight laws in $\Sp_1$ along some
      deterministic vanishing-mesh sequence;
\item $\mathbf W^{\mathcal D}$ converges in probability in $\rho_\alpha$
      for some $1/3<\alpha<1/2$ along some such sequence.
\end{enumerate}
When they hold, the local areas have a continuous $\Sp_1$-valued version,
and $\mathbf W=(W,A+\tfrac12W^{\otimes2})$ is a canonical geometric
projective rough path. Along every deterministic vanishing-mesh sequence,
\begin{equation}\label{eq:projective-rough-convergence}
 \rho_\alpha(\mathbf W^{\mathcal D},\mathbf W)\longrightarrow0
 \quad\text{in }L^r(\Omega),\qquad
 1/3<\alpha<1/2,\quad1\le r<\infty.
\end{equation}
In particular, $A_{\mathcal D}\to A$ in every finite
$L^r(\Omega;\Sp_1)$. If $\mathcal H(Q)=\infty$, then
$\|A_{\mathcal D}(W)\|_\pi\to\infty$ in probability along every
deterministic vanishing-mesh sequence.
\end{theorem}

\begin{proof}
\emph{The Wiener model and time projection.}
For $v=\sum_i\sqrt{q_i}e_i$, $2T_Q$ is the difference of the strict
triangular parts of $v\otimes v$, and $2T_Q+v\otimes v$ is triangular
with diagonal $(q_i)$
\cite{KwapienPelczynski1970,AleksandrovPeller2022,ArlinskiiTsekanovskii2006}.
Every even principal minor of $2T_Q$ on $I$ equals $\prod_{i\in I}q_i$:
the sign matrix has Pfaffian one, by first-row expansion and
$\sum_{j=2}^{2k}(-1)^j=1$. Odd minors vanish. Finite determinant
expansion and the $\Sp_1$ and $\Sp_2$ compression limits give
\begin{equation}\label{eq:even}
 Z_{T_Q}(z)=\sum_{k\ge0}4^{-k}e_{2k}(q)z^{2k}
 =\tfrac12[\det(I+zQ/2)+\det(I-zQ/2)].
\end{equation}
Let $J$ on $L^2(0,1)$ have kernel $J(s,t)=\tfrac12\operatorname{sgn}(t-s)$,
put $\eta_0=\omega_J$ and $Gf=\int_0^1f\,dW$.
Let $R_{\mathcal D}$ project onto the time-cell indicators and set
$E_{\mathcal D}=\E[\,\cdot\mid\Delta_1W,\ldots,\Delta_mW]$.
The compression $J_{\mathcal D}=R_{\mathcal D}JR_{\mathcal D}$ has matrix
\eqref{eq:model} in the normalized cell basis, with weights $h_j=t_j-t_{j-1}$;
its Gaussian evaluation is \eqref{eq:polygonal-area}. Only diagonal cells are lost:
\begin{equation}\label{eq:hs-partition-error}
 \|J-J_{\mathcal D}\|_2^2=\tfrac14\sum_jh_j^2,\qquad
 \E\|A-A_{\mathcal D}\|_2^2=\tfrac12e_2(q)\sum_jh_j^2.
\end{equation}
Continuity of $\Gamma_2$ and It\^o isometry identify
$\omega_A=\Gamma_2\eta_0$. The even part of
$\prod_j(1+zh_j/2)$ tends to $\cosh(z/2)$ as $|\mathcal D|\to0$;
therefore $Z_J(z)=\cosh(z/2)$ and
\begin{equation}\label{eq:timecoeff}
 (2k)!\left\|\eta_0^{\wedge k}/k!\right\|^2=4^{-k}.
\end{equation}
Equations \eqref{eq:mean} and \eqref{eq:even} prove \eqref{eq:D}
and identify the model in Theorem~\ref{thm:main}.

Gaussian conditioning applies $R_{\mathcal D}$ to each time slot
\cite[Chapter~5]{Janson1997}, giving
\begin{equation}\label{eq:local-area-conditioning}
 E_{\mathcal D}(W_{s,t},A_{s,t})
 =(W^{\mathcal D}_{s,t},A^{\mathcal D}_{s,t})
 \quad\text{in }H\oplus\Sp_2.
\end{equation}
Indeed, for cells $I_a$, put $\theta_a=|(s,t)\cap I_a|/|I_a|$.
The local kernel $\tfrac12\mathbf1_{(s,t)^2}(u,v)\operatorname{sgn}(v-u)$
has cell average $\tfrac12\theta_a\theta_b\operatorname{sgn}(b-a)$,
exactly the area coefficient for the polygonal pieces $\theta_a\Delta_aW$.

\emph{Covariance criteria and spectral counts.}
At $z=2ix$, \eqref{eq:even} equals
$\prod_i(1+x^2q_i^2)^{1/2}\cos\theta_Q(x)$.
The phase increases from zero to $\pi\operatorname{rank}Q/2$, so its
positive singular-value pairs satisfy
\begin{equation}\label{eq:covariance-phase}
 \theta_Q((2\sigma_j)^{-1})=\pi(j-\tfrac12),\qquad
 N_{T_Q}(t)=2\left\lceil\theta_Q((2t)^{-1})/\pi-\tfrac12\right\rceil.
\end{equation}
Since $\arctan x\asymp\min(1,x)$, this proves the deterministic part
of \eqref{eq:general-area-counts}.
Finite rank is immediate; in infinite rank, layer cake and
$N_{T_Q}(t)\asymp M_Q(t)$ give
\[
 \int_0^S t^{p-1}M_Q(t)\,dt
 =\frac1p\sum_iq_i^p+\sum_iq_i\int_{q_i}^S t^{p-2}\,dt.
\]
The last sum is finite for $p>1$, equals $\mathcal H(Q)$ for $p=1$,
and is finite exactly when $\sum_iq_i^p<\infty$ for $0<p<1$.
The common dilation transfers these criteria to all stages for $0<p<2$;
membership for $p\ge2$ is automatic. This proves \eqref{eq:classification}
simultaneously for all $p$ and stages.

For exact stage counts in infinite rank, let
$d_t=N_{T_Q}(t)\to\infty$. The bound
$0\le x\theta_Q'(x)\le\theta_Q(x)$ gives the normalized threshold-band
bound $O(\varepsilon+d_t^{-1})$ in \eqref{eq:threshold-band}, and
$N_{T_Q}(tu)\le1+(d_t+1)\max(1,u^{-1})$.
Hence \eqref{eq:resolvent-from-count} gives
$\mathcal R_{T_Q}(1/t)=O(d_t)$; Lemma~\ref{lem:count-transfer}
proves \eqref{eq:general-area-counts} uniformly in the stages.
Theorem~\ref{thm:main}(c), applied to the infinite-rank $J$, proves
\eqref{eq:area-concentration} and its uniform $L^p$ form. The count
formulations follow from $N_Q(t)/M_Q(t)=r_t/(1+r_t)$,
where $r_t=tN_Q(t)/\tau_Q(t)$.

\emph{The projective criterion.}
The trace-class case of \eqref{eq:classification} gives
\textup{(i)}$\Leftrightarrow$\textup{(ii)}.
If entropy is infinite, then $\|A\|_1=\infty$ almost surely.
The extended trace norm is lower semicontinuous on $\Sp_2$, being the
supremum of continuous finite singular-value sums. Equation
\eqref{eq:hs-partition-error} gives $A_{\mathcal D}\to A$ in probability
in $\Sp_2$. Portmanteau applied to its closed trace-norm sublevel sets
yields $\Pbb(\|A_{\mathcal D}\|_1\le M)\to0$ for every $M<\infty$.
This excludes \textup{(iii)} and, by terminal evaluation, \textup{(iv)}.

Assume now finite entropy. Equation \eqref{eq:mean-schatten} gives
$\E\|A\|_1\le\|T_Q\|_1<\infty$.
For a Banach-valued second-chaos variable $Y$,
positivity of the Ornstein--Uhlenbeck semigroup gives hypercontractivity;
interpolation yields
$\|Y\|_{L^4}\le3\|Y\|_{L^2}$,
$\|Y\|_{L^2}\le9\E\|Y\|$, and consequently
\begin{equation}\label{eq:area-moments}
 \|A\|_{L^r(\Omega;\Sp_1)}
 \le9\max\{1,r-1\}\|T_Q\|_1\qquad(1\le r<\infty).
\end{equation}
For covariance-coordinate projections $P_N$ exhausting the support
of $Q$, $P_NAP_N\to A$ almost surely in $\Sp_1$; apply the preceding
finite-dimensional bounds and Fatou.
Brownian scaling gives
$\|A_{s,t}\|_{L^r(\Sp_1)}\le C_r|t-s|$ and
$\|W_{s,t}\|_{L^r(H)}\le C_r|t-s|^{1/2}$.
Finite compressions give Chen's identity on dyadic points for
$\mathbb W=A+\tfrac12W^{\otimes2}$. Rough-path Kolmogorov
\cite[Section~1.5 and Theorem~3.1]{FrizHairer2020} gives a continuous version with
$\|W\|_\beta+\|A\|_{2\beta}^{1/2}$ in every finite $L^r$,
for $1/3<\beta<1/2$.

\emph{Conditional approximation of the lift.}
The inclusion $\Sp_1\hookrightarrow\Sp_2$ commutes with Bochner
conditional expectation, so \eqref{eq:local-area-conditioning} now also
holds in $H\oplus\Sp_1$.
For $1/3<\alpha<\beta<1/2$, the Brownian-measurable field
$X_\alpha=(W_{s,t}/|t-s|^\alpha,A_{s,t}/|t-s|^{2\alpha})$, extended by
zero on the diagonal, belongs to
$L^r(\Omega;C(\Delta;H\oplus\Sp_1))$ for all $1\le r<\infty$,
where $\Delta=\{0\le s\le t\le1\}$.
Since $W^{\mathcal D}\to W$ uniformly almost surely, every bounded
continuous cylinder $f$ satisfies
$\|E_{\mathcal D}f(W)-f(W)\|_r
\le2\|f(W^{\mathcal D})-f(W)\|_r\to0$.
Cylinder density and contractivity extend this to Brownian-measurable
$L^r(\Omega;B)$ for every separable Banach space $B$.
Apply this to $X_\alpha$; a countable dense set identifies its conditional
expectation with the scaled fields in \eqref{eq:local-area-conditioning},
giving
\begin{equation}\label{eq:conditioning-approx}
 E_{\mathcal D}X_\alpha\longrightarrow X_\alpha
 \quad\text{in }L^r(\Omega;C(\Delta;H\oplus\Sp_1)).
\end{equation}
H\"older's inequality and
\[
 \|\widetilde W^{\otimes2}-W^{\otimes2}\|_{2\alpha}
 \le(\|\widetilde W\|_\alpha+\|W\|_\alpha)
       \|\widetilde W-W\|_\alpha
\]
prove \eqref{eq:projective-rough-convergence} and terminal $L^r$
convergence. Smooth approximation of polygonal lifts in their
finite-dimensional spans, followed by an almost surely convergent
subsequence, proves geometricity.
\end{proof}

\section*{Declaration of AI use}
The author used OpenAI's GPT models interactively in developing proofs and refining the presentation. The author is responsible for the final content of the manuscript.

\end{document}